\documentclass[reqno,oneside,11pt,a4paper]{amsart}%
\usepackage{hyperref}
\usepackage{amsthm}
\usepackage{amscd}
\usepackage{amsmath,amssymb}
\usepackage{dcolumn}
\usepackage{amsmath}
\usepackage{amsfonts}
\usepackage{amssymb}
\usepackage{graphicx}
\usepackage{algorithm}
\usepackage[noend]{algorithmic}
\usepackage{todonotes}%
\newcolumntype{d}{D{.}{.}{-1}}
\theoremstyle {plain}
\newtheorem {theorem}{Theorem}[section]

\newtheorem {lemma}[theorem]{Lemma}

\theoremstyle {definition}
\newtheorem {definition}[theorem]{Definition}

\theoremstyle {remark}
\newtheorem {remark}[theorem]{Remark}

\newtheorem {example}[theorem]{Example}
\DeclareMathOperator{\im}{im}

\DeclareMathOperator{\LM}{LM}

\DeclareMathOperator{\Quot}{Q}

\begin{document}
\title{The use of Bad Primes in Rational Reconstruction}
\author{Janko B\"ohm}
\address{Fachbereich Mathematik\\
Technical University Kaiserslautern\\
Postfach 3049\\
67653 Kaiserslautern\\
Germany}
\email{boehm@mathematik.uni-kl.de}
\author{Wolfram Decker}
\address{Fachbereich Mathematik\\
Technical University Kaiserslautern\\
Postfach 3049\\
67653 Kaiserslautern\\
Germany}
\email{decker@mathematik.uni-kl.de}
\author{Claus Fieker}
\address{Fachbereich Mathematik\\
Technical University Kaiserslautern\\
Postfach 3049\\
67653 Kaiserslautern\\
Germany}
\email{fieker@mathematik.uni-kl.de}
\author{Gerhard Pfister}
\address{Fachbereich Mathematik\\
Technical University Kaiserslautern\\
Postfach 3049\\
67653 Kaiserslautern\\
Germany}
\email{pfister@mathematik.uni-kl.de}

\begin{abstract}
A standard method for finding a rational number from its values modulo a
collection of primes is to determine its value modulo the product of the
primes via Chinese remaindering, and then use Farey sequences for rational
reconstruction. Successively enlarging the set of primes if needed, this
method is guaranteed to work if we restrict ourselves to \textquotedblleft
good\textquotedblright\ primes. Depending on the particular application,
however, there may be no efficient way of identifying good primes.

In the algebraic and geometric applications we have in mind, the final result
consists of an a priori unknown ideal (or module) which is found via a
construction yielding the (reduced) Gr\"obner basis of the ideal. In this
context, we discuss a general setup for modular and, thus, potentially
parallel algorithms which can handle \textquotedblleft bad\textquotedblright%
\ primes. A key new ingredient is an error tolerant algorithm for rational
reconstruction via Gaussian reduction.

\end{abstract}
\date{\today }
\maketitle

\section{Introduction}

Rational reconstruction, in conjunction with Chinese remaindering, provides a
standard way of obtaining results over the rational numbers from results in
characteristic $p>0$. This is of particular use in the design of parallel
algorithms and in situations where the growth of intermediate results matters.
Classical applications are the computation of polynomial greatest common
divisors (see \cite{W81,E95}) and Gr\"{o}bner bases (see \cite{arnold}). The
goal of the modular Gr\"{o}bner basis algorithm presented in \cite{arnold} is
to compute the Gr\"{o}bner basis of an ideal \emph{already known}. That is,
the ideal is given by a finite set of generators. In contrast, there are
constructions which yield \emph{a priori unknown} ideals by finding their
(reduced) Gr\"obner bases. Prominent examples are the computation of
normalization and the computation of adjoint curves (see \cite{BDLPSS,BDLSadj}%
). Here, for the purpose of modularization, we expect that the respective
construction applies to a given set of input data in characteristic zero as
well as to the
modular values of the input data. In such a situation, problems may arise
through the presence of \textquotedblleft bad\textquotedblright\ primes of
different types.

Usually, a first step to resolve these problems is to show that the bad primes
are \textquotedblleft rare\textquotedblright. Then the different types of bad
primes are addressed. For example, a prime for which it is a priori clear that
the modular construction does not work will simply be rejected. Depending on
the application, however, there may be bad primes which can only be detected a
posteriori, that is, after the true characteristic zero result has been found.
For such an application, we must ensure that the reconstruction algorithm used
will return the correct rational number even in the presence of bad primes. In
this note, we derive such an algorithm. Based on this algorithm, we describe a
general scheme for computing ideals (or modules) in the algebraic and
geometric applications we have in mind, addressing in particular how various
types of bad primes are handled. This scheme has already been successfully
used in the aforementioned papers \cite{BDLPSS,BDLSadj}.

To begin, in Section \ref{sec reconstruction of a rational number}, we recall
the classical approach to rational reconstruction which is based on the
lifting of modular to rational results by computing Farey preimages via
Euclidean division with remainder. Section \ref{sec types of bad primes}
contains a short discussion of the different types of bad primes. In Section
\ref{sec reconstruction with bad primes}, we present the new lifting algorithm
which is based on Gaussian reduction, and discuss the resulting error tolerant
reconstruction algorithm. Finally, in Section
\ref{sec applications in algebraic geometry}, we present our general scheme
for applications in commutative algebra and algebraic geometry. We finish by
giving an explicit example which involves a bad prime that can only be
detected a posteriori.

To summarize, we focus on the presentation of a general setup for modular
computations based on error tolerant rational reconstruction. We do not
discuss implementation details or performance questions. In fact, for the
applications we have in mind, the time used for rational reconstruction can be
neglected in view of the more involved parts of the respective algorithms.

\section{Reconstruction of a single rational
number\label{sec reconstruction of a rational number}}

We describe the reconstruction of a single unknown number $x\in\mathbb{Q}$.
In practical applications, this number will be a coefficient of an object to
be computed in characteristic $0$ (for example, a vector, polynomial, or
Gr\"{o}bner basis). Here, we suppose that we are able to verify the
correctness of the computed object (in some applications by a comparably easy
calculation, in others by a bound on the size of the coefficients).

We use the following notation: Given an integer $N\geq2$ and a number
$x=\frac{a}{b}\in\mathbb{Q}$ with $\gcd(a,b)=1$ and $\gcd(b,N)=1$, the
\emph{value of $x$ modulo $N$} is
\[
x_{N}:=\left(  \frac{a}{b}\right)  _{N}:=(a+N\mathbb{Z} )(b+N\mathbb{Z}%
)^{-1}\in\mathbb{Z}/N\mathbb{Z}.
\]
We also write $x\equiv r\bmod N$ if $r\in\mathbb{Z}$ represents $x_{N}$.

In what follows, we suppose that in the context of some application, we are
given an algorithm which computes the value of the unknown number
$x\in\mathbb{Q}$ modulo any prime $p$, possibly rejecting the prime. For
reference purposes, we formulate this in the black box type Algorithm
\ref{alg 1}.

\begin{algorithm}[H]
\caption{Black Box Algorithm $x\bmod p$}
\label{alg 1}
\begin{algorithmic}[1]
\REQUIRE A prime number $p$.
\ENSURE {\tt false} or an integer $0\leq s\leq p-1$ such that $x\equiv s\bmod p$.
\end{algorithmic}
\emph{Assumption:} There are only finitely many primes $p$ where the return value is {\tt false}.
\end{algorithm}

Once the values of $x$ modulo each prime in a sufficiently large set of primes
$\mathcal{P}$ have been computed, we may find $x$ via a lifting procedure
consisting of two steps: First, use Chinese remaindering to obtain the value
of $x$ modulo the product $N:=\prod_{p\in\mathcal{P}}p$. Second, compute the
preimage of this value under the \emph{Farey rational map} which is defined as follows.

For an integer $B>0$, set%
\[
F_{B}=\left\{  \frac{a}{b}\in\mathbb{Q}\mid\gcd(a,b)=1\text{, }0\leq a\leq
B\text{, }0<\left\vert b\right\vert \leq B\right\}  \text{,}%
\]
and for $m\in\mathbb{Z}/N\mathbb{Z}$, let%
\[
\mathbb{Q}_{N,m}=\left\{  \frac{a}{b}\in\mathbb{Q}\mid\gcd(a,b)=1\text{, }%
\gcd(b,N)=1\text{, }\left(  \frac{a}{b}\right)  _{N}=m\right\}
\]
be the set of rational numbers whose value modulo $N$ is $m$. Then
$\mathbb{Q}_{N}={%
{\textstyle\bigcup\nolimits_{m=0}^{N-1}}
}\mathbb{Q}_{N,m}$ is a subring of $\mathbb{Q}$ with identity. If $B$ is an
integer with $B\leq\sqrt{(N-1)/2}$, then the \emph{Farey map}%
\[
\varphi_{B,N}:F_{B}\cap\mathbb{Q}_{N}\rightarrow\mathbb{Z}/N\mathbb{Z}%
,\quad\frac{a}{b}\mapsto\left(  \frac{a}{b}\right)  _{N},
\]
is well--defined and injective (but typically not surjective). To obtain the
injective map with the largest possible image for a given $N$, we tacitly
suppose in what follows that $B$ is chosen as large as possible for $N$.

Algorithm \ref{algo:lift} below will return $\varphi_{B,N}^{-1}(\overline{r})$
if $\overline{r}$ is in the image of the Farey map, and \texttt{false}
otherwise (see, for example, \cite{KG83,W81,WGD, CE94}).
\begin{algorithm}[h]
\caption{Farey Preimage}
\label{algo:lift}
\begin{algorithmic}[1]
\REQUIRE Integers $N\geq2$ and $0\leq r\leq N-1$.
\ENSURE {\tt false} or a rational number $\frac{a}{b}$ with $\gcd(a,b)=1$, $\gcd(b,N)=1$,
$\frac{a}{b}\equiv r\bmod N$,  $0\leq a \leq\sqrt{(N-1)/2}$, $0<\left\vert b\right\vert \leq \sqrt{(N-1)/2}$.
\STATE $(a_{0},b_{0}):=(N,0)$, $(a_{1},b_{1}):=(r,1)$, $i:=-1$
\WHILE {$2a_{i+2}^{2} > N-1$ }
\STATE $i:=i+1$
\STATE divide $a_i$ by $a_{i+1}$ to find $q_{i},a_{i+2},b_{i+2}$ such that
\[
(a_{i},b_{i})=q_{i}(a_{i+1},b_{i+1})+(a_{i+2},b_{i+2})
\]
and $0\leq a_{i+2}<a_{i+1}$
\ENDWHILE
\IF {$2b_{i+2}^{2}\leq N-1$ \AND $\gcd(a_{i+2},b_{i+2})=1$} \RETURN $\frac{a_{i+2}}{b_{i+2}}$
\ENDIF
\RETURN {\tt false}
\end{algorithmic}
\end{algorithm}

Combining Algorithm \ref{algo:lift} with Chinese remaindering as indicated
above, we get the classical reconstruction Algorithm \ref{algo:main}.

\begin{algorithm}[h]
\caption{Reconstruction of a Rational Number}
\label{algo:main}
\begin{algorithmic}[1]
\REQUIRE Algorithm \ref{alg 1} and a way to verify that a
computed number equals $x$.
\ENSURE $x$.
\STATE $N:=1$, $r:=0$
\STATE $p:=2$
\LOOP
\STATE \label{algo:step:4} let $s$ be the return value of Algorithm
\ref{alg 1} applied to $p$
\IF {$s=$ {\tt false}}
\STATE continue with Step \ref{algo:step:3}
\ENDIF
\STATE find $1=eN+fp$ and set $r:=rfp+seN$, $N:=Np$
\STATE  let $y$ be the return value of Algorithm
\ref{algo:lift} applied to $N$ and $r$
\IF {$y=$ {\tt false}}
\STATE continue with step \ref{algo:step:3}
\ENDIF
\IF {$y=x$}
\RETURN $y$
\ENDIF
\STATE \label{algo:step:3} $p:=$NextPrime$(p)$
\ENDLOOP
\end{algorithmic}
\end{algorithm}

Note that Algorithm \ref{algo:main} works correctly since we suppose that our
Black Box Algorithm \ref{alg 1} either returns \texttt{false} or a correct
answer. For most applications, however, there exist primes $p$ which are bad
in the sense that the algorithm under consideration returns a wrong answer
modulo $p$ which can only be detected a posteriori. In this note, we show that
if there are only finitely many such primes, they can just be ignored. More
precisely, we show that in Algorithm \ref{algo:main}, we may call the black
box type Algorithm \ref{alg 1prime} below instead of Algorithm \ref{alg 1},
provided we then call the lifting Algorithm \ref{algo:lift new} from Section
\ref{sec reconstruction with bad primes} instead of Algorithm \ref{algo:lift}.

\begin{algorithm}[h]
\caption{Black Box Algorithm $x\bmod p$ With Errors}
\label{alg 1prime}
\begin{algorithmic}[1]
\REQUIRE A prime number $p$.
\ENSURE {\tt false} or an integer $0\leq s\leq p-1$.
\end{algorithmic}
\emph{Assumption:} There are only finitely many primes $p$ where either
the return value is {\tt false} or $x\not\equiv s\bmod p$.
\end{algorithm}

\section{Types of bad primes for modular algorithms
\label{sec types of bad primes}}

In this section, we suppose that we are given an algorithm implementing a
construction which applies in any characteristic, together with a set of input
data over the rationals.

We call a prime $p$ \emph{good} for the given algorithm and input data if the
algorithm applied to the modulo $p$ values of the input returns the reduction
of the characteristic zero result. Otherwise, the prime is called \emph{bad.}
In what follows, to make the discussion in the subsequent sections more clear,
we specify various types of bad primes and describe their influence on the
design of algorithms.

A prime $p$ is bad of \emph{type-1} if the modulo $p$ reduction of the
characteristic zero input is either not defined or for obvious reasons not a
valid input for the algorithm. For example, if the input is a polynomial,
type-1 primes arise from prime factors of a denominator of a coefficient.
Type-1 primes are harmless with regard to rational reconstruction as they can
be detected and, thus, rejected from the beginning, at no additional cost.

For a bad prime $p$ of \emph{type-2}, it turns out only in the course of the
construction that the computation in characteristic $p$ cannot be carried
through. For example, an algorithm for inverting a matrix will not work for a
prime dividing the determinant. Since, typically, the determinant is not
known, the failure will only turn out eventually. Type-2 primes waste
computation time, but with regard to rational reconstruction they are detected
before the Chinese remainder step and do, thus, not influence the final lifting.

Consider an invariant whose characteristic zero value coincides with the
characteristic $p$ value for all good primes $p$, and suppose that this value
is known a priori. Moreover assume that the algorithm computes this invariant
at some point of the construction.
Then a prime $p$ is bad of \emph{type-3} if the value in characteristic $p$
differs from the expected one. Like type-2 primes, bad primes of type-3 waste
computation time for computing modular results which will then be discarded,
but do not influence the final lifting. Examples of possible invariants are
the dimension or the degree or of a variety. Note that the computation of an
invariant for detecting a type-3 prime may be expensive. Dropping the
computation of the invariant in the design of the algorithm, if possible, will
turn a type-3 prime into a prime of different type. This includes primes of
type-5 below.

Now suppose that some invariant associated to the modular output is computed
by the algorithm, and that the a priori \emph{unknown} characteristic zero
value of this invariant coincides with the characteristic $p$ value for all
good primes $p$.
Then a prime is bad of \emph{type-4} if this invariant does not have the
correct value. Such a prime cannot be detected a priori. However, if there are
only finitely many such primes, they can be eliminated with arbitrarily high
probability by a majority vote over several primes. Type-4 bad primes may
occur, for example, in modular Gr\"obner basis computations, where we use the
leading ideal as an invariant for voting.
Handling type-4 primes is expensive not only since we waste computation time,
but also since we have to store the modular results for the majority vote.
Again, these primes are eventually detected and, hence do not enter the
Chinese remainder step.

Bad primes other than those discussed so far are called bad primes of
\emph{type-5}. This includes primes which cannot be discovered by any known
means without knowledge of the characteristic zero result. Example
\ref{ex type-5} below shows that type-5 bad primes indeed occur in algebraic
geometry. Type-5 bad primes enter the Chinese remainder step and are, thus,
present during the final lifting process. Considering Algorithm
\ref{algo:main}, calling black box Algorithm \ref{alg 1prime} instead of
Algorithm \ref{alg 1}, we will be in a situation where always either Algorithm
\ref{algo:lift} or the comparison \texttt{$x=y$} will return \texttt{false}.
As a result, the algorithm will not terminate.

Due to their nature, bad primes create hardly ever a \emph{practical} problem.
Typically, there are only very few bad primes for a given instance, and these
will not be encountered if the primes used are chosen at random. On the other
hand, for some of the modern algorithms in commutative algebra, we have no
\emph{theoretical} argument eliminating type-5 bad primes. Hence, we need
error tolerant reconstruction, which ensures termination provided there are
only finitely many bad primes.

\section{Reconstruction with bad
primes\label{sec reconstruction with bad primes}}

To design our error tolerant reconstruction algorithm, we turn rational
reconstruction into a lattice problem.

To begin with, given an integer $N\geq2$, we define the subset $C_{N}%
\subseteq\mathbb{Z}/N\mathbb{Z}$ of elements applied to which Algorithm
\ref{algo:lift new} below will return a rational number (and not
\texttt{false}). Let $C_{N}$ be the set of all $\overline{r}\in\mathbb{Z}%
/N\mathbb{Z}$ such that there are integers $u,v\in\mathbb{Z}$ with $u\geq0$,
$v\neq0$, and $\gcd(u,v)=1$ which satisfy the following condition:
\begin{equation}
\label{equ:def-CN}%
\begin{tabular}
[c]{c}%
there is an integer $q\geq1$ with $q|N$ and such that\\
\vspace{-0.2cm}\\
$u^{2}+v^{2}<\frac{N}{q^{2}}$\quad and\quad$u\equiv vr\bmod\frac{N}{q}$.\\
\end{tabular}
\end{equation}

In Lemma \ref{lem same ratio} below, we will prove that the rational number
$\frac{u}{v}=\frac{uq}{vq}\in\mathbb{Q}$ is uniquely determined by condition
\eqref{equ:def-CN}. Hence, we have a well--defined map
\[
\psi_{N}:C_{N} \rightarrow\mathbb{Q}.
\]
Note that the image of the Farey map $\varphi_{B,N}$, with $B=\left\lfloor
\sqrt{(N-1)/2}\right\rfloor $, is contained in $C_{N}$: If $\overline{r}
\in\im( \varphi_{B,N})$, then $\varphi_{B,N}^{-1}(\overline{r})$ satisfies
Condition \eqref{equ:def-CN} with $q=1$. Furthermore, $\varphi_{B,N}%
^{-1}(\overline{r}) = \psi_{N}(\overline{r})$.

Typically, the inclusion $\im( \varphi_{B,N})\subseteq C_{N}$ is strict:

\begin{example}
For $N=2\cdot13$, we have $B=3$, hence%
\[
\im( \varphi_{B,N}) =\left\{  \overline{0},\overline{1},\overline{2}%
,\overline{3},\overline{8},\overline{9},\overline{17},\overline{18}%
,\overline{23},\overline{24},\overline{25}\right\}  \text{,}%
\]
and the rational numbers which can be reconstructed by Algorithm
\ref{algo:lift} are the elements of%
\[
F_{B}\cap\mathbb{Q}_{N}=\left\{  0,\pm1,\pm2,\pm3,\pm\frac{1}{3},\pm\frac
{2}{3}\right\}  \text{.}%
\]
On the other hand,
\[
C_{N}=\left\{  \overline{r}\mid0\leq r\leq25\text{, }r\neq5,21\right\}  ,
\]
and Algorithm \ref{algo:lift new} will reconstruct the rational numbers in
\[
\psi_{N}(C_{N}) = \left\{  0,\pm1,\pm2,\pm3,\pm4,\pm\frac{1}{2},\pm\frac{1}%
{3},\pm\frac{2}{3},\pm\frac{4}{3}\right\}  .
\]
Note that the denominator of $\frac{1}{2}=\psi_{N}(7)=\psi_{N}(20)$ is not
coprime to $N$. In both cases, $q=2$: We have $1\equiv2\cdot7\bmod 13$ and
$1\equiv2\cdot20\bmod 13$.

\end{example}

Now, fix $0\leq r\leq N-1$ such that $\overline{r}\in C_{N}$, and consider the
lattice $\Lambda=\Lambda_{N,r}:=\langle(N,0),(r,1)\rangle$ of discriminant
$N$. Let $u,v,q$ correspond to $\overline{r}$ as in Condition
\eqref{equ:def-CN}. Then $(uq,vq)\in\Lambda_{N,r}$.

\begin{lemma}
\label{lem same ratio} With notation as above, all $(x,y)\in\Lambda$ with
$x^{2}+y^{2}<N$ are collinear. That is, they define the same rational number
$\frac{x}{y}$.

\end{lemma}

\begin{proof}
Let $\lambda=(x,y),\, \mu=(c,d)\in\Lambda$ be vectors with $x^{2}+y^{2}%
,c^{2}+d^{2} < N$. Then $y\mu-d\lambda=(yc-xd,0)\in\Lambda$, so $N|(yc-xd)$.
Since $|yc-xd|<N$ by Cauchy--Schwarz, we get $yc=xd$, as claimed.
\end{proof}

Next, consider integers $N^{\prime},M\geq2$, with $\gcd(M,N^{\prime})=1$, and
such that $N=N^{\prime}M$. Let $a\geq0,b\neq0$ be integers such that
$\gcd(b,N^{\prime})=1$, and let $a\equiv bs\bmod N^{\prime}$, with $0\leq
s\leq N^{\prime}-1$. Let $0\leq t\leq M-1$ be another integer, and let $0\leq
r\leq N-1$ be the Chinese remainder lift satisfying $r\equiv s\bmod N^{\prime
}$ and $r\equiv t\bmod M$. In practical applications, we think of $N^{\prime}$
and $M$ as the product of good and bad primes, respectively. By the following
lemma, Algorithm \ref{algo:lift new} below applied to $N$ and $r$ will return
$a/b$ independently of the possibly \textquotedblleft wrong
result\textquotedblright\ $t$, provided that $M\ll N^{\prime}$.

\begin{lemma}
\label{lemma2} With notation as above,
suppose that $(a^{2}+b^{2})M<N^{\prime}$. Then, for all $(x,y)\in
\Lambda=\langle(N,0) , (r,1)\rangle$ with $(x^{2}+y^{2})<N$, we have $\frac
{x}{y}=\frac{a}{b}$. Furthermore, if $\gcd(a,b)=1$ and $(x,y)$ is a shortest
nonzero vector in $\Lambda$, we also have $\gcd(x,y)|M$.
\end{lemma}

\begin{proof}
From $a\equiv bs\bmod N^{\prime}$, we get $a-bs=k_{1}N^{\prime}$ for some
$k_{1}$. Moreover, $s\equiv r\bmod N^{\prime}$ gives $r=s+k_{2}N^{\prime}$.
Now $(aM,bM)-bM(r,1)=(aM-brM,0)$ and $aM-brM=M(a-br)=M(a-b(s+k_{2}N^{\prime
}))=M(a-bs)-k_{2}bN=k_{1}N-k_{2}bN$, thus $(aM,bM)\in\Lambda$. Since
$(a^{2}+b^{2})M<N^{\prime}$, Lemma \ref{lem same ratio} gives $\frac{a}%
{b}=\frac{aM}{bM}=\frac{x}{y}$ for all $(x,y)\in\Lambda$ such that
$(x^{2}+y^{2})<N$.

For the second statement, write $A:=(aM,bM)$ and $X:=(x,y)$. By Lemma
\ref{lem same ratio}, there is a $\lambda=\frac{s}{t}\in\mathbb{Q}$, with
$\gcd(s,t)=1$, and such that $\lambda X=A$. The Euclidean Algorithm gives
integers $e,f$ with $et+sf=1$, hence%
\[
\frac{X}{t}=(et+sf)\frac{X}{t}=eX+fA\in\Lambda\text{.}%
\]
Since $X$ is a shortest vector in the lattice, it follows that $t=\pm1$, hence
$A=\pm sX$. Since $\gcd(a,b)=1$, we conclude that $\gcd(x,y)| M$.
\end{proof}

The use of Lemma \ref{lemma2} is twofold. From a theoretical point of view, it
allows us to ignore type-5 bad primes in the design of modular algorithms --
as long as there are only finitely many of them. This makes the design of
modular algorithms much simpler. From a practical point of view, it allows us
to avoid expensive computations of invariants to eliminate bad primes of any
type.
Moreover, factorizing the $\gcd$ of the components of a shortest lattice
element can help us to identify bad primes (see Example \ref{Ex:bad-primes}
below).

Lemma \ref{lemma2} yields the correctness of both the new lifting Algorithm
\ref{algo:lift new} and the resulting reconstruction Algorithm \ref{algo:main}%
, calling black box Algorithm \ref{alg 1prime} instead of Algorithm
\ref{alg 1}, and lifting Algorithm \ref{algo:lift new} instead of Algorithm
\ref{algo:lift}. In applications, the termination can be based either on the
knowledge of a priori bounds on the height of $\frac{x}{y}$ or on an a
posteriori verification of the result. It should be mentioned that both
methods are used: some problems allow for easy verification, while others
yield good bounds.

\begin{algorithm}[h]
\caption{Error Tolerant Lifting}
\label{algo:lift new}
\begin{algorithmic}[1]
\REQUIRE  Integers $N\geq2$ and $0\leq r\leq N-1$.
\ENSURE $\psi_N (\overline r )$ if $\overline r\in C_N$ and {\tt false} otherwise.
\STATE $(a_{0},b_{0}):=(N,0)$, $(a_{1},b_{1}):=(r,1)$, $ i:=-1$
\REPEAT
\STATE $i=i+1$
\STATE set
\[
q_{i}=\left\lfloor \frac{\langle (a_{i},b_{i}), (a_{i+1},b_{i+1})\rangle}{\Vert (a_{i+1}, b_{i+1})
\Vert^2}
\right\rceil
\]
\STATE set
\[
(a_{i+2},b_{i+2})=(a_{i},b_{i})-q_{i}(a_{i+1},b_{i+1})
\]
\UNTIL {$a_{i+2}^{2} + b_{i+2}^{2}\geq a_{i+1}^{2} + b_{i+1}^{2}$}
\IF {$a_{i+1}^{2} + b_{i+1}^{2}<N$}
\RETURN $\frac{a_{i+1}}{b_{i+1}}$
\ELSE
\RETURN {\tt false}
\ENDIF
\end{algorithmic}
\end{algorithm}

\begin{remark}
Algorithm \ref{algo:lift new}, which is just a special case of Gaussian
reduction, will always find a shortest vector in the lattice generated by
$(N,0)$ and $(r,1)$. Moreover, $b_{i}\neq0$ for all $i>0$ since in every step
the vector $(a_{i},b_{i})$ gets shorter and, hence, cannot be equal to $(N,0)$.

Even though Algorithm \ref{algo:lift new} looks more complicated than
Algorithm \ref{algo:lift}, the bit--complexity of both algorithms is the same:
$O(\log^{2}N)$. See the discussion in \cite[Section 3.3]{stehle}.
\end{remark}

\begin{example}
\label{Ex:bad-primes} We reconstruct the rational number $\frac{13}{12}$ using
the modulus%
\[
N=38885=5\cdot7\cdot11\cdot101\text{.}%
\]
With notation as above, $a=13$, $b=12$, $r=22684$, and the Farey bound is%
\[
B=\left\lfloor \sqrt{(N-1)/2}\right\rfloor =139\text{.}%
\]
Algorithm \ref{algo:lift} applied to this data will correctly return
$\frac{13}{12}$. Similarly for Algorithm \ref{algo:lift new} which generates
the sequence%
\begin{align*}
(38885,0)  &  =2\cdot(22684,1)+(-6483,-2),\\
(22684,1)  &  =-3\cdot(-6483,-2)+(3235,-5),\\
(-6483,-2)  &  =2\cdot(3235,-5)+(-13,-12),\\
(3235,-5)  &  =-134\cdot(-13,-12)+(1493,-1613).
\end{align*}

Now, bad primes will enter the picture. Consider the Chinese remainder
isomorphism%
\[
\chi:\mathbb{Z}/5\mathbb{Z\times Z}/7\mathbb{Z\times Z}/11\mathbb{Z\times
Z}/101\mathbb{Z\rightarrow Z}/38885\mathbb{Z}\text{.}%
\]
The preimage of $\overline{r}=\left(  \frac{13}{12}\right)  _{N}$ is
\[
\chi^{-1}(\overline{r})=(\overline{4},\overline{4},\overline{2},\overline
{60}).
\]
That is, $\overline{r}$ is the solution to the simultaneous congruences%
\begin{align*}
x  &  \equiv4\operatorname{mod}5\\
x  &  \equiv4\operatorname{mod}7\\
x  &  \equiv2\operatorname{mod}11\\
x  &  \equiv60\operatorname{mod}101\text{.}%
\end{align*}
If we make $101$ a bad prime by changing the congruence $x\equiv
60\operatorname{mod}101$ to $x\equiv61\operatorname{mod}101$, we obtain%
\[
\chi(\overline{4},\overline{4},\overline{2},\overline{61})=\overline
{16524}\text{.}%
\]
Algorithm \ref{algo:lift new} then computes%
\begin{align*}
(38885,0)  &  =2\cdot(16524,1)+(5837,-2),\\
(16524,1)  &  =3\cdot(5837,-2)+(-987,7),\\
(5837,-2)  &  =6\cdot(-987,7)+(-85,40),\\
(-987,7)  &  =10\cdot(-85,40)+(-137,393)\text{.}%
\end{align*}
Hence the output $\frac{85}{-40}=\frac{17}{8}\neq\frac{13}{12}$ is not the
desired lift. The reason for this is that $101$ is not small enough compared
to its cofactor in $N$. Algorithm \ref{algo:lift}, on the other hand, returns
\texttt{false} since the reduction process will also terminate with $(85,-40)$
and these numbers are not coprime.
Note that Algorithm \ref{alg generic lift for ideals} in Section
\ref{sec applications in algebraic geometry} below will detect an incorrect
lift either by the procedure \textsc{pTest} (with a very high probability) or
the subsequent verification step over the rationals (carried through only if
\textsc{pTest} returns \texttt{true}). As a consequence, in both cases, the
set of primes will be enlarged (without discarding previous results).
Eventually, the good primes will outweigh the bad ones and Algorithm
\ref{algo:lift new}, when called from Algorithm
\ref{alg generic lift for ideals}, will return the correct lift.

For example, replace the congruence $x\equiv4\operatorname{mod}7$ by
$x\equiv2\operatorname{mod}7$, so that
\[
\chi(\overline{4},\overline{2},\overline{2},\overline{60})=\overline
{464}\text{.}%
\]
Then Algorithm \ref{algo:lift new} yields%
\begin{align*}
(38885,0)  &  =84\cdot(464,1)+(-91,-84),\\
(464,1)  &  =-3\cdot(-91,-84)+(191,-251)\text{,}%
\end{align*}
and terminates with the correct lift%
\[
\frac{91}{84}=\frac{13}{12}\text{.}%
\]
Algorithm \ref{algo:lift}, on the other hand, will again return \texttt{false}
since the reduction also terminates with the numbers $(91,84)$ which are not coprime.

Since
\[
(13^{2}+12^{2})\cdot7<5\cdot11\cdot101,
\]
Lemma \ref{lemma2} shows that $7$ is small enough compared to its cofactor in
$N$. Hence, the wrong result $2$ modulo the bad prime $7$ does not influence
the result of the lift. In fact, all other possible congruences modulo $7$
will lead to the same output. Note that the bad prime can be detected as
$\gcd(91,84,N)=7$. Furthermore, note that in the example the lifting process
involving the bad prime requires fewer steps than the process relying on good
primes only.
\end{example}

\section{A setup for applications in algebra and
geometry\label{sec applications in algebraic geometry}}

In this section, we discuss a general computational setup for applications in
commutative algebra and algebraic geometry which requires error tolerance. A
setup of this type occurs, for example, when computing normalization or when
computing adjoint curves. See \cite{BDLPSS,BDLSadj} and Example
\ref{exa:norm-adj} below.

To begin, fix a global monomial ordering $>$ on the monoid of monomials in the
variables $X=\{X_{1},\ldots,X_{n}\}$. Consider the polynomial rings
$W={\mathbb{Q}}[X]$ and, given an integer $N\geq2$, $W_{N}=(\mathbb{Z}%
/N\mathbb{Z})[X]$. If $T\subseteq W$ or $T\subseteq W_{N}$ is a set of
polynomials, then denote by $\operatorname{LM}(T):=\{\operatorname{LM}(f)\mid
f\in T\}$ its set of leading monomials. If $f\in W$ is a polynomial such that
$N$ is coprime to any denominator of a coefficient of $f$, then its
\emph{reduction modulo $N$} is the polynomial $f_{N}\in W_{N}$ obtained by
mapping each coefficient $x$ of $f$ to $x_{N}$ as described in Section
\ref{sec reconstruction of a rational number}. If $H=\{h_{1},\dots
,h_{s}\}\subset W$ is a Gr\"{o}bner basis such that $N$ is coprime to any
denominator in any $h_{i}$, set $H_{N}=\{(h_{1})_{N},\dots,(h_{s})_{N}\}$. If
$J\subseteq W$ is any ideal, its \emph{reduction modulo $N$} is the ideal
\[
J_{N}=\left\langle f_{N}\mid f\in J\cap\mathbb{Z}[X]\right\rangle \subseteq
W_{N}\text{.}%
\]

\noindent\emph{\textsc{Notation:} From now on, let $I\subset W$ be a fixed
ideal. }

\begin{remark}
\label{rem:reduction-practical} For practical purposes, $I$ is given by a set
of generators. Fix one such set $f_{1},\ldots,f_{r}$. Then we realize the
reduction of $I$ modulo a prime $p$ via the following equality which holds for
all but finitely many primes $p$:
\[
I_{p}=\left\langle (f_{1})_{p},...,(f_{r})_{p}\right\rangle \subseteq W_{p}.
\]
More precisely, when running the modular Algorithm
\ref{alg generic lift for ideals} described below, we incorporate the
following: if one of the $(f_{i})_{p}$ is not defined (that is, $p$ is bad of
type-1 for the given set of generators), we reject the prime\footnote{Note
that rescaling to integer coefficients is not helpful: reducing the rescaled
generators may yield the wrong leading ideal. See Remark \ref{rmk arnold}.}.
If all $(f_{i})_{p}$ are defined, we work with the ideal on the right hand
side instead of $I_{p}$. Note that is possible to detect primes with
$I_{p}\neq\left\langle (f_{1})_{p},...,(f_{r})_{p}\right\rangle $ (which are
hence of type-3). Indeed, $I_{p}$ can be found using Gr\"obner bases (see
\cite[Cor. 4.4.5]{AL} and \cite[Lem. 6.1]{arnold}). However, we suggest to
skip this computation: finitely many bad primes will not influence the result
if we use error tolerant rational reconstruction as in Algorithm
\ref{algo:lift new}.

\end{remark}

To simplify our presentation in what follows, we will systematically ignore
the primes discussed in Remark \ref{rem:reduction-practical}. We suppose that
we are given a construction which associates to $I$ a uniquely determined
ideal $U(0)\subseteq W$, and to each reduction $I_{p}$, with $p$ a prime
number, a uniquely determined ideal $U(p)\subseteq W_{p}$, where we make the
following assumption:

\vspace{0.2cm}

\vspace{0.2cm}

\noindent\emph{\textsc{Assumption:} We ask that \mbox{$U(0)_{p}=U(p)$} for all
but finitely many $p$. }

\vspace{0.2cm}

We write $G(0)$ for the uniquely determined reduced Gr\"{o}bner basis of
$U(0)$, and $G(p)$ for that of $U(p)$. In the applications we have in mind, we
wish to construct the unknown ideal $U(0)$ from a collection of its
characteristic $p$ counterparts $U(p)$. Technically, given a finite set of
primes $\mathcal{P}$, we wish to construct $G(0)$ by computing the $G(p)$,
$p\in\mathcal{P}$, and lifting the $G(p)$ coefficientwise to characteristic
zero. Here, to identify Gr\"obner basis elements corresponding to each other,
we require that $\operatorname{LM}(G(p))=\operatorname{LM}(G(q))$ for all
$p,q\in\mathcal{P}$. This leads to condition (1b) below:

\begin{definition}
\label{defnLucky} With notation as above, we define:

\begin{enumerate}
\item A prime number $p$ is called \emph{lucky} if the following hold:

\begin{enumerate}
\item $U(0)_{p}=U(p)$ and

\item $\operatorname{LM}(G(0))=\operatorname{LM}(G(p))$.
\end{enumerate}

\noindent Otherwise $p$ is called \emph{unlucky}.

\item If $\mathcal{P}$ is a finite set of primes, set
\[
N^{\prime}=\prod_{p\in\mathcal{P}\text{ lucky}}p \hspace{0.5cm} \text{and}%
\hspace{0.5cm} M=\prod_{p\in\mathcal{P}\text{ unlucky}}p\text{.}
\]
Then $\mathcal{P}$ is called \emph{sufficiently large} if
\[
N^{\prime}> (a^{2}+b^{2}) \cdot M
\]
for all coefficients $\frac{a}{b}$ of polynomials in $G(0)$ (assume
$\mathop{gcd}(a,b)=1$).

\end{enumerate}
\end{definition}

Note that a prime $p$ violating condition (1a) is of type-5, while (1b) is a
type-4 condition.

\begin{remark}
\label{rmk arnold} A modular algorithm for the fundamental task of computing
Gr\"{o}bner bases is presented in \cite{arnold} and \cite{IPS}. In contrast to
our situation here, where we wish to find the ideal $U(0)$ by computing its
reduced Gr\"{o}bner basis $G(0)$, Arnold's algorithm starts from an ideal
which is already given. If $p$ is a prime number, $J\subseteq W$ is an ideal,
$H(0)$ is the reduced Gr\"{o}bner basis of $J$, and $H(p)$ is the reduced
Gr\"{o}bner basis of $J_{p}$, then $p$ is lucky for $J$ in the sense of Arnold
if $\operatorname{LM}(H(0))=\operatorname{LM}(H(p))$. It is shown in
\cite[Thm. 5.12 and 6.2]{arnold} that if $J$ is homogeneous and $p$ is lucky
for $J$ in this sense, then $H(0)_{p}$ is well--defined and equal to $H(p)$.
Furthermore, by \cite[Cor. 5.4 and Thm. 5.13]{arnold}, all but finitely many
primes are Arnold--lucky for a homogeneous $J$. Using weighted homogenization
as in the proof of \cite[Thm. 2.4]{IPS}, one shows that these results also
hold true in the non--homogeneous setup.
\end{remark}

\begin{example}
Consider the ideal $J=\langle\frac{\partial{f}}{\partial{x}},\frac{\partial
{f}}{\partial{y}}\rangle$, where $f=x^{5}+y^{11}+xy^{9}+x^{3}y^{9}$. The
leading terms of the lexicographical Gr\"obner basis with integral coprime
coefficients are as follows:
\begin{align*}
&  264627y^{39}+ \dots,\\
&  12103947791971846719838321886393392913750065060875xy^{8}-\dots,\\
&  40754032969602177507873137664624218564815033875x^{4}+\dots.
\end{align*}
Hence, the Arnold unlucky primes for $J$ are
\[
3, 5, 11, 809, 65179, 531264751, 431051934846786628615463393.
\]

\end{example}

With respect to our notion of lucky as introduced in Definition
\ref{defnLucky}.(1), we first observe:

\begin{lemma}
\label{lem finite} The set of unlucky primes is finite.
\end{lemma}

\begin{proof}
By our general assumption, $U(0)_{p}=U(p)$ for all but finitely many primes
$p$. Given a prime $p$ such that $U(0)_{p}=U(p)$, we have $\operatorname{LM}%
(G(0))=\operatorname{LM}(G(p))$ if $p$ does not divide the denominator of any
coefficient of any polynomial occurring when testing whether $G(0)$ is a
Gr\"{o}bner basis using Buchberger's criterion. The result follows.
\end{proof}

The type-5 condition (1a) cannot be checked a priori: We compute $G(p)$ and,
thus, $U(p)$ on our way, but $U(0)_{p}$ is only known to us after $G(0)$ and,
thus, $U(0)$ have been found. However, this is not a problem if we apply error
tolerant rational reconstruction: the finitely many bad primes leading to an
ideal $U(p)$ different from $U(0)_{p}$ will not influence the final result:

\begin{lemma}
If $\mathcal{P}$ is a sufficiently large set of primes satisfying condition
(1b), then the reduced Gr\"{o}bner bases $G(p)$, $p\in\mathcal{P}$, lift via
Algorithm \ref{algo:lift new} to the reduced Gr\"{o}bner basis $G(0)$.
\end{lemma}

\begin{proof}
If a prime $p$ satisfies (1b), then $p$ is Arnold--lucky for $U(0)$. Hence, as
remarked above, $G(0)_{p}=G(p)$. Since $\mathcal{P}$ is sufficiently large, by
Lemma \ref{lemma2}, the coefficients of the Chinese remainder lift $G(N)$,
$N={\textstyle\prod\nolimits_{p\in\mathcal{P}}}$, have a lift to
characteristic zero. By the uniqueness statement of Lemma \ref{lemma2}, this
lift coincides with $G(0)$.

\end{proof}

Lemma \ref{lem finite} guarantees, in particular, that a sufficiently large
set $\mathcal{P}$ of primes satisfying condition (1b) exists. So from a
theoretical point of view, the idea of finding $G(0)$ is now as follows:
Consider such a set $\mathcal{P}$, compute the reduced Gr\"{o}bner bases
$G(p)$, $p\in\mathcal{P}$, and lift the results to $G(0)$ as described above.

From a practical point of view, we face the problem that we cannot test a
priori whether $\mathcal{P}$ is sufficiently large and satisfies condition
(1b). However, to remedy the situation, we can proceed in the following
randomized way:

First, fix an integer $t\geq1$ and choose a set of $t$ primes $\mathcal{P}$ at
random. Second, compute $\mathcal{GP}=\{G(p)\mid p\in\mathcal{P}\}$, and use a
majority vote with respect to the type-4 condition (1b):

\vspace{0.2cm}

\noindent\emph{\textsc{deleteByMajorityVote:} Define an equivalence relation
on $\mathcal{P}$ by setting $p\sim q:\Longleftrightarrow\LM(G(p))=\LM(G(q))$.
Then replace $\mathcal{P}$ by an equivalence class of largest
cardinality\footnote{When computing the cardinality, take Remark
\ref{rem:cardinality} into account.}, and change $\mathcal{GP}$ accordingly.}

\vspace{0.2cm}

Now all Gr\"obner bases in $\mathcal{GP}$ have the same leading ideal. Hence,
we can lift the Gr\"{o}bner bases in $\mathcal{GP}$ to a set of polynomials
$G\subseteq W$. Furthermore, if $\mathcal{P}$ is suffciently large, all primes
in $\mathcal{P}$ satisfy (1b). However, since we do not know whether
$\mathcal{P}$ is sufficiently large, a final verification in characteristic
zero is needed. As this may be expensive, especially if $G\neq G(0)$, we first
perform a test in positive characteristic:

\vspace{0.2cm}

\noindent\emph{\textsc{pTest:} Randomly choose a prime number $p\notin
\mathcal{P}$ such that $p$ does not divide the numerator or denominator of any
coefficient occurring in a polynomial in $G$ or $\left\{  {f_{1}}%
,\ldots,{f_{r}}\right\} $. Return \texttt{true} if $G_{p}=G(p)$, and
\texttt{false} otherwise.}

\vspace{0.2cm}

If \textsc{pTest} returns \texttt{false}, then $\mathcal{P}$ is not
sufficiently large or the extra prime chosen in \textsc{pTest} is bad. In this
case, we enlarge the set $\mathcal{P}$ by $t$ primes not used so far, and
repeat the whole process. If \textsc{pTest} returns \texttt{true}, however,
then most likely $G=G(0)$. Hence, it makes sense to verify the result over the
rationals. If the verification fails, we enlarge $\mathcal{P}$ and repeat the process.

We summarize this approach in Algorithm \ref{alg generic lift for ideals}
(recall that we ignore the primes from Remark \ref{rem:reduction-practical} in
our presentation).

\begin{algorithm}[h]
\caption{Reconstruction of an Ideal}
\label{alg generic lift for ideals}
\begin{algorithmic}[1]
\REQUIRE An algorithm  to compute $G(p)$ from $I_{p}$, for each prime $p$,
and a way of verifying that a computed Gr\"obner basis equals $G(0)$.
\ENSURE The Gr\"{o}bner basis $G(0)$.
\STATE choose a list $\mathcal{P}$ of random primes
\STATE $\mathcal{GP}=\emptyset$
\LOOP
\FOR {$p\in\mathcal{P}$}
\STATE compute $G(p)\subseteq W_{p}$
\STATE $\mathcal{GP}=\mathcal{GP}\cup\left\{  G(p)\right\}  $
\ENDFOR
\STATE $(\mathcal{P},\mathcal{GP})=\text{\emph{\textsc{deleteByMajorityVote}}%
}(\mathcal{P},\mathcal{GP})$
\STATE lift $\mathcal{GP}$ to $G\subseteq W$ via Chinese remaindering and
Algorithm \ref{algo:lift new}
\IF {the lifting succeeds and \emph{\textsc{pTest}}$(I,G,\mathcal{P})$}
\IF {$G=G(0)$}
\RETURN $G$
\ENDIF
\ENDIF
\STATE enlarge $\mathcal{P}$ with primes not used so far
\ENDLOOP
\end{algorithmic}
\end{algorithm}

\begin{remark}
\label{rem:cardinality}

If Algorithm \ref{alg generic lift for ideals} requires more than one round of
the loop, we have to use a weighted cardinality count in
\textsc{{deleteByMajorityVote}}: when enlarging $\mathcal{P}$, the total
weight of the elements already present must be strictly smaller than the total
weight of the new elements. Otherwise, though highly unlikely in practical
terms, it may happen that only unlucky primes are accumulated.

\end{remark}

\begin{remark}
Our lifting process works since reduced Gr\"{o}bner bases are uniquely
determined. In practical terms, however, there is often no need to reduce the
Gr\"{o}bner bases involved: it is only required that the construction
associating the Gr\"{o}bner bases to $I$ and its reductions yields uniquely
determined results.
\end{remark}

\begin{remark}
We may allow that the given construction does not work for finitely many
primes $p$ (which are then bad of type-2). In this case, the respective primes
will be rejected.
\end{remark}

\begin{remark}
Depending on the particular implementation of the construction, type-3 primes
(in addition to those considered in Remark \ref{rem:reduction-practical}) may
occur. In such a situation, it is often cheaper to rely on error tolerance
rather than spending computation time to detect these primes.
\end{remark}

\begin{example}
\label{exa:norm-adj} If $K$ is any field, and $I\subseteq K[X]$ is a prime
ideal, the \emph{normalization} $\overline{A}$ of the domain $A=K[X]/I$ is the
integral closure of $A$ in its field of fractions $\Quot(A)$. If $K$ is
perfect, the normalization algorithm given in \cite{GLS} will find a
\textquotedblleft valid denominator\textquotedblright\ $d\in A$ and an ideal
$U\subseteq A$ such that $\frac{1}{d}U=\overline{A}\subseteq\Quot(A)$. In
fact, $U$ is uniquely determined if we fix $d$. In practical terms, $d$ and
$U$ are a polynomial and an ideal in $K[X]$, respectively. If $K=\mathbb{Q}$,
we can apply the modular version of the algorithm (see \cite{BDLPSS}). This
version relies on choosing a \textquotedblleft universal
denominator\textquotedblright\ $d$ which is used over the rationals as well as
in finite characteristic. Given a prime number $p$, it may then happen that
$I_{p}$ is not a prime ideal (a type-2 condition), that the leading ideal of
$I_{p}$ does not coincide with that of $I$ (a type-3 condition), that $d_{p}$
is not defined (a type-1 condition), or that $d_{p}$ is not a valid
denominator (a type-2 condition). In accordance with the general setup, the
numerator ideal $U(p)$ is obtained by computing the Gr\"{o}bner basis $G(p)$,
and $\operatorname{LM}(G(0))=\operatorname{LM}(G(p))$ and $U(0)_{p}=U(p)$ are
type-4 and type-5 conditions, respectively.

\end{example}

The normalization algorithm mentioned in the previous Example
\ref{exa:norm-adj} finds $\overline{A}$ by successively enlarging $A$ in form
of an endomorphism ring of a so called test ideal $J\subseteq A$. For
practical purposes, the radical of the Jacobian ideal is used for $J$. The
following example shows that, for the algorithm computing the radical of the
Jacobian, bad primes $p$ exist which satisfy the type-4 condition (1b) but
violate the type-5 condition (1a) in Definition \ref{defnLucky}. In
particular, these primes cannot be eliminated by a majority vote on the
leading ideal.

\begin{example}
\label{ex type-5}We construct a sextic curve $C=V(I)\subset\mathbb{P}%
_{\mathbb{C}}^{2}$, given by an ideal $I=\left\langle f\right\rangle
\subset\mathbb{Q}[x,y,z]$, such that $5$ is a bad prime of type-5 for
computing the radical of the singular locus of $C$. The basic idea is to
construct a curve which has two double points in characteristic $0$, which
coincide when reducing modulo $5$, while one additional double point appears.

For%
\[
L=\left\langle y,\text{ }x-4z\right\rangle ^{2}\cap\left\langle y,\text{
}x+6z\right\rangle ^{2}\subseteq\mathbb{Q}[x,y,z]
\]
the reduced Gr\"{o}bner basis with respect to the degree reverse
lexicographical ordering is%
\[
G=\left(  y^{2},\text{ }(x-4z)(x+6z)y,\text{ }(x-4z)^{2}(x+6z)^{2}\right)
\text{.}%
\]
Note, that both $L$ and%
\[
L_{5}=\left\langle y^{2},\text{ }(x+z)^{2}y,\text{ }(x+z)^{4}\right\rangle
\subseteq\mathbb{F}_{5}[x,y,z]
\]
have the same leading monoid $\left\langle y^{2},x^{2}y,x^{4}\right\rangle $.

Writing generators of%
\[
M=L_{5}\cap\left\langle y,\text{ }x-z\right\rangle ^{2}\subset L_{5}%
\]
in terms of the Gr\"{o}bner basis of $L_{5}$ yields the representation%
\[
M=\left\langle y^{2},\text{ }(x-z)\cdot(x+z)^{2}y,\text{ }(x^{2}%
+3xz+z^{2})\cdot(x+z)^{4}\right\rangle \text{.}%
\]
Now consider a generic homogeneous element of degree $6$ in%
\[
\left\langle G_{1},\text{ }(x-z)\cdot G_{2},\text{ }(x^{2}+3xz+z^{2})\cdot
G_{3}\right\rangle \subset\mathbb{Q}[x,y,z]\text{.}%
\]
For practical puposes, a random element will do, for example,%
\[
f=x^{6}+y^{6}+7x^{5}z+x^{3}y^{2}z-31x^{4}z^{2}-224x^{3}z^{3}+244x^{2}%
z^{4}+1632xz^{5}+576z^{6}\text{.}%
\]

For the ideal $I=\left\langle f\right\rangle $ the prime $5$ is bad of type-5
with respect to the algorithm\vspace{-0.01cm}
\[
I\mapsto\sqrt{I+\operatorname{Jac}(I)}\text{,}%
\]
where $\operatorname{Jac}(I)$ denotes the Jacobian ideal of $I$: First note,
that no coefficient of $f$ is divisible by $5$. In particular,
$\operatorname{LM}(I)=\left\langle x^{6}\right\rangle =\operatorname{LM}%
(I_{5})$, so $5$ is Arnold--lucky for $I$. We compute%
\begin{align*}
U(0) &  =\sqrt{I+\operatorname{Jac}(I)}=\left\langle y,x-4z\right\rangle
\cap\left\langle y,x+6z\right\rangle \text{,}\\
U(5) &  =\sqrt{I_{5}+\operatorname{Jac}(I_{5})}=\left\langle y,x^{2}%
-z^{2}\right\rangle =\left\langle y,x-z\right\rangle \cap\left\langle y,\text{
}x+z\right\rangle \text{,}%
\end{align*}
and
\[
U(0)_{5}=\left\langle y,(x+z)^{2}\right\rangle \text{.}%
\]
Hence%
\[
\operatorname{LM}(U(0))=\left\langle y,x^{2}\right\rangle =\operatorname{LM}%
(U(5))\text{,}%
\]
but%
\[
U(0)_{5}\neq U(5)\text{.}%
\]

\end{example}

\vspace{0.1in} \noindent\emph{Acknowledgements}. We would like to thank the
referees who made valuable suggestions to improve the presentation of this paper.

\end{document}